%% file: ConvergenceLawPermutationClasses.tex
\documentclass[
final
]{dmtcs-episciences}

\usepackage[utf8]{inputenc}
\usepackage{subfigure}
\usepackage[parfill]{parskip}    % Activate to begin paragraphs with an empty line rather than an indent
\usepackage{xspace,csquotes}
\usepackage{amssymb,amsmath}   
\usepackage{bm}%,enumitem}
\usepackage{tikz}
\usetikzlibrary[positioning,patterns]
\usetikzlibrary{calc}

\usepackage[round]{natbib}
\DeclareMathOperator{\SR}{SR}

\newcommand{\TOTO}{\ensuremath{\mathsf{TOTO}}\xspace} 
\newcommand{\TOOB}{\ensuremath{\mathsf{TOOB}}\xspace}

\def\PP{\mathbb P}
\def\eps{\varepsilon}

\renewcommand{\th}{^{\mbox{\scriptsize th}}}

\newtheorem{lemma}{Lemma}
\newtheorem{proposition}[lemma]{Proposition}
\newtheorem{theorem}[lemma]{Theorem}

\DeclareMathOperator{\Av}{Av}
\DeclareMathOperator{\Cat}{Cat}
\DeclareMathOperator{\Id}{Id}
\DeclareMathOperator{\qdepth}{qd}

\title{A logical limit law for $231$-avoiding permutations}

\author[M. Albert, M. Bouvel, V. Féray and M. Noy]{Michael Albert\affiliationmark{1}  \and Mathilde Bouvel\affiliationmark{2}  \and Valentin Féray\affiliationmark{3}
\and Marc Noy\affiliationmark{4} }
  
  \affiliation{
  Department of Computer Science, University of Otago, Dunedin, New Zealand\\
  Université de Lorraine, CNRS, Inria, LORIA, F-54000 Nancy, France\\
  Université de Lorraine, CNRS, IECL, F-54000 Nancy, France\\
  Department of Mathematics and Institute of Mathematics, Universitat Politècnica de Catalunya, Barcelona, Spain}

\begin{document}
\publicationdata{vol. 26:1, Permutation Patterns 2023}{2024}{1}{10.46298/dmtcs.11751}{2023-08-22; 2023-08-22; 2023-12-08}{2023-12-22}
 
\maketitle

\begin{abstract}~ 

We prove that the class of 231-avoiding permutations satisfies a logical limit law,
i.e.~that for any first-order sentence $\Psi$, in the language of two total orders, the probability $p_{n,\Psi}$ that
a uniform random 231-avoiding permutation of size $n$ satisfies $\Psi$
admits a limit as $n$ is large.
Moreover, we establish two further results about the behavior and value of $p_{n,\Psi}$:
(i) it is either bounded away from $0$, or decays exponentially fast;
(ii) the set of possible limits is dense in $[0,1]$.
Our tools come mainly from analytic combinatorics and singularity analysis.
\end{abstract}

\section{Introduction}

\subsection{Background}

For any model of random combinatorial structures (e.g., permutations, graphs, \dots),
a natural problem is to compute the asymptotic probability that they satisfy a property of interest.
A step further consists in considering this problem for general sets of properties.
To this end, it is useful to use {\em finite model theory}.
In this context, the combinatorial objects are seen as models of some logical theory
(e.g., graphs are finite sets with a binary symmetric anti-reflexive relation).
Then finite model theory allows one to define a whole hierarchy of properties on our object:
(existential) first-order properties, (existential/monadic) second order properties, and so on. 
In this paper, we will be interested in first-order properties,
which are the ones that can be written using only quantifiers on elements (and not on sets), equalities between elements,
the relation(s) of the language (e.g.~two elements being neighbours in graphs)
and boolean operations; see below for an informal discussion on the expressive
power of first-order logic, and
Section~\ref{sec:logic} for a formal definition.

Let us consider a sequence of random combinatorial structures $ s_n$, for example graphs or permutations,
seen as models of a logical theory.
We say that $s_n$
satisfies a (first-order) logical limit law\footnote{In some texts, the name ``convergence law'' is used instead of ``logical limit law''.}
if, for any first-order sentence $\Psi$,
the probability that $s_n$ satisfies $\Psi$, denoted $\mathbb P( s_n \models \Psi)$,
 has a limit as $n$ tends to $+\infty$.
If, additionally, the limit is always $0$ or $1$,
then $s_n$ satisfies a (first-order) 0-1 law.

This formal logic approach in discrete probability started around 1970
with the seminal works of \cite{GKLT} and~\cite{FaginFiniteModels}, 
who independently proved that  a uniform random simple graph $G_n$ 
with $n$ vertices satisfies a 0-1 law. 
More generally, for the Erd\H{o}s-Rényi model $G(n,p)$ 
with $p \sim n^\alpha$ ($\alpha \in (0,1)$), 
a remarkable result of~\cite{ShelahSpencer_ZeroOneLaws} states
that $G(n,p)$ satisfies a 0-1 law 
if and only if $\alpha$ is irrational.
Recently, a collection of results has appeared regarding
existence or non-existence of 0-1/logical limit laws for uniform random graphs
taken in a given graph class; see, {\em e.g.},~\cite{LimitLaws_MinorClosed}
and \cite{convergenceLaw_PerfectGraphs}.

For permutations, fewer results of this type are available.
Two different ways of seeing permutations as models of some logical theory
have been described by \cite{albert2020two}.
We will focus on the one called \TOTO ({\em Theory Of Two Orders}), where permutations
are seen as finite sets, endowed with a pair of linear orders $(A,<_P,<_V)$ 
(comparing respectively the positions and the values of elements of the permutation;
see Section~\ref{sec:logic} for details).
With respect to \TOTO,
 it is known that uniform random permutations 
 $ \sigma_n$ do not satisfy a logical limit law,
i.e.~there exist first-order properties $\Psi$ such that
$\mathbb P( \sigma_n \models \Psi)$ does not have a limit (and actually,
can be taken to oscillate between $0$ and $1$);
see \cite{ConvLaw_Permutations} 
(note that this reference does not use the permutation language,
but considers the equivalent setting of pairs of linear orders)
or \cite{muller2023logical} (where the more general setting of Mallows random permutations
is considered).
On the opposite, uniform layered permutations\footnote{A permutation is layered
if it is an increasing sequence of decreasing runs (of arbitrary length).}
do satisfy a logical limit law \cite{BraunfeldKukla}.
\medskip

{\em Expressive power of the first-order logic on permutations.}
To make things more concrete, let us explain informally which kind of properties
$\Psi$ can be expressed as a first-order property in the \TOTO logic.
 More details can be found in \cite{albert2020two}.
The containment of a given pattern, either in the classical or consecutive sense,
is a first-order property. One can also consider the generalizations 
considered in the literature
(vincular, bivincular, meshed, barred, decorated patterns); see \cite{albert2020two}
for details. This contains many classical statistics on permutations: left-to-right maxima (or other types of records), adjacencies (two elements which are consecutive both in positions and values), indecomposable blocks,~\dots~
For each of these statistics, one can express the fact that a permutation contains exactly/at most/at least $k$ of those (for any fixed $k$), and any boolean combination of these properties. For example, that a permutation contains at most nine inversions and exactly two adjacencies is a first order property. 
One can also express properties of the first/last/maximum/minimum of the permutation, {\em e.g.}, that the minimum of a permutation occurs before its maximum. On the other hand, that a permutation contains an even/odd number of inversions is not a first order property. It is also impossible to consider statistics that compares elements of the domain to elements of the co-domain of the permutations,
such as existence of fixed points, exceedances, \dots

\subsection{Main result}
We recall that a permutation $\sigma$ contains a permutation $\pi$ as a pattern
if there is a subsequence of $\sigma$ which is order-isomorphic to $\pi$. 
For instance,  the permutation $6473512$ contains $231$ as a pattern:
indeed, its subsequence $472$ is order-isomorphic to $231$.
When $\sigma$ and $\pi$ are interpreted as models of \TOTO, 
this just means that $\pi$ is (isomorphic to) a submodel of $\sigma$.
For a given pattern $\pi$, the set of permutations avoiding $\pi$ is denoted $\Av(\pi)$, and for any integer $n$, $\Av_n(\pi)$ denotes the set of permutations of size $n$ in $\Av(\pi)$.
Sets of permutations of the form $\Av(\pi)$, called (principal) permutation classes,
have been widely studied in the enumerative combinatorics literature (see \cite{VatterSurvey} for a survey)
and more recently also from the probabilistic point of view
(see, {\em e.g.},~\cite{PHC_Dichotomy} and references therein).
In this article we consider one of the simplest nontrivial cases, namely the class $\Av(231)$.
We prove a logical limit law for a uniform random 231-avoiding permutation $\sigma_n$ of size $n$.
We also provide two additional results on the possible asymptotic behavior
of $\PP(\sigma_n \models \Psi)$, where $\Psi$ a first order sentence on permutations.

\begin{theorem}
  \label{thm:ConvLaw_Av231}
  For each $n\ge 1$, let $\sigma_n$ be a uniform random 231-avoiding permutation of size $n$. Then $\sigma_n$ satisfies a logical limit law.
  Moreover, 
  \begin{enumerate}%[label=(\roman*)]
    \item  if $\Psi$ is a first order sentence on permutations,
      then either $\lim \PP(\sigma_n \models \Psi)>0$,
      or there exists $\eps=\eps(\Psi)<1$ such that
      \[ \PP(\sigma_n \models \Psi)= O(\eps^n); \]
    \item the set of limiting probabilities
  $\big\{\! \lim \PP(\sigma_n \models \Psi), \Psi\text{\ FO formula} \big\}$ is dense in $[0,1]$.
  \end{enumerate}
\end{theorem}
The proof of the logical limit law is inspired by a paper of \cite{woods1997coloring}, proving a logical limit law for uniform random nonplane trees
in monadic second order logic.
It relies on techniques of analytic combinatorics,
in particular on a general result on the type of singularity
of polynomial systems of equations 
(commonly known as the Drmota--Lalley--Woods theorem).
Item 1 above is a consequence of Woods' proof technique.
For the description of the set of limiting probabilities, we exhibit and combine
sufficiently many simple events, whose asymptotic limiting probability is straightforward to compute. We then use a lemma of \cite{KakeyaSubsums}, indicating when the set of subsums of a given convergent series is dense in the relevant interval.
The second part of our result can be compared to the results of~\cite{LimitLaws_MinorClosed} and~\cite{LimitingProbabilities_Gnp},
where the set of limiting probabilities is described for some random graph models.
At least for $\Av(231)$, the picture is simpler in the setting of permutations.

We note that a logical limit law for another permutation class,
namely the class of {\em layered} permutations has
been recently established by \cite{BraunfeldKukla}.
The techniques are different from the ones in the present paper.
Interestingly, none of these techniques is easily adapted to $\Av(321)$
(even though they have the same number of elements of each size,
the classes  $\Av(231)$ and $\Av(321)$ are known to have different structures
in many ways). We do not know whether $\Av(321)$ admits a logical limit law or not\footnote{
{\em Note added in proof.} A logical limit law for the class $\Av(321)$ has been established
by \cite{Oezdemir2023cv_law_321}, using an \enquote{infinite-dimensional version of the Perron--Frobenius theorem}.
}.
More generally, it is not known whether there exists a proper permutation class
for which the logical limit law fails; see \cite[Section 5]{BraunfeldKukla}.
\medskip

{\em What about the \TOOB logic?}
As mentioned above, in~\cite{albert2020two}, two different ways of seeing
permutations as models in some logical theory were considered.
In this paper, we only consider one of them, \TOTO.
The other framework, \TOOB ({\em Theory Of One Bijection}),
regards permutation as maps from a finite set $A$ to itself, 
and consists of a single relation expressing that an element $x$ is sent to $y$.
In \TOOB, one can express condition on cycles of fixed lengths,
for example that a permutation has more than three fixed points 
and at least one cycle of length at least 10.
It is however impossible to compare values of the elements;
in particular conjugate permutations are indistinguishable for \TOOB.

The expressibility of \TOOB is in some sense poorer than that of \TOTO,
and the question of logical limit law is essentially equivalent to the 
convergence of short cycle counts; see \cite[Section 4]{muller2023logical}.
In particular, it is easy to prove that uniform random permutations 
satisfy a logical limit law for \TOOB; see again \cite[Section 4]{muller2023logical}.
We are not aware of logical limit law results for uniform pattern-avoiding permutations 
for \TOOB.
As said above, this amounts to studying their short cycle counts.
Fixed points in 213-avoiding (resp.~123-avoiding and 321-avoiding\footnote{Unlike \TOTO, the \TOOB framework is not invariant 
by the action of all symmetries of the square acting on permutation matrices,
and considering 123-avoiding permutations is not equivalent to 321-avoiding permutations.
It is however still invariant by symmetries along diagonals, 
so that considering 213-avoiding or 132-avoiding permutations are equivalent problems.
As far as we are aware, fixed points in 231-avoiding, or equivalently in 312-avoiding permutations, have not been studied.}) permutations have been studied
in \cite{HoffmanBrownian2,Hoffman2019fixed321}, where convergence in distribution results are proved.
We might expect similar results for the number of cycles of length $k$ (for any fixed $k$)
which would imply a logical limit law in these cases, 
but proving it would require a significant amount of work.

\section{Preliminaries}

\subsection{Permutations as models of a logical theory and first order sentences}
\label{sec:logic}
We present here briefly the logical theory $\TOTO$ (theory of two orders).
Details and general references for finite model theory can be found in \cite{albert2020two}.
The {\em signature} of the theory consists of two binary relations $<_P$ and $<_V$.
The {\em axioms} of the theory specify that $<_P$ and $<_V$ are linear (or total) order relations.
A model in the theory is then a set $A$ endowed with two linear orders, also denoted $<_P$ and $<_V$.
We will only be interested here in finite models.
As explained in \cite{albert2020two}, isomorphism classes of finite models
are naturally indexed by permutations.
Indeed, think of a permutation as its permutation matrix, where $1$s are replaced by points
and $0$s by empty cells. Then $\sigma$ can be identified with the sets $A^\sigma$ of points,
together with the relations $<^\sigma_P$ and $<^\sigma_V$,
comparing respectively the $x$ and $y$-coordinates of points
(or in other terms, their positions and their values
in the permutation). See Fig.~\ref{fig:permutation_TOTO} for an example.

\begin{figure}
\input{pic.txt}
%    \[\begin{array}{c}
%    \includegraphics[height=2cm]{DiagramPermutationEtiquettes}
%  \end{array}
%  \quad
%  \begin{array}{c}
%    \text{\scriptsize $A <_P B <_P C <_P D <_P E;$}\\
%    \text{\scriptsize $B <_V E <_V A <_V C <_V D.$}
%  \end{array}\]
%  \caption{A permutation and its associated linear orders $<_P$ and $<_V$}
\caption{A permutation in matrix form with associated linear orders
$A <_P B <_P C <_P D <_P E$, and
$B <_V E <_V A <_V C <_V D$.}
\label{fig:permutation_TOTO}
\end{figure}
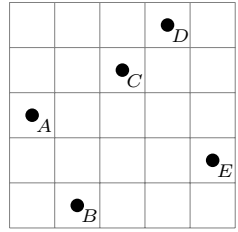

We now define first-order formulas and sentences.
Take an infinite set $\{x,y,z,\dots\}$ of variables.
Atomic formulas are constructed by taking variables and connecting them 
with a relation of the signature or with the equality symbol.
In our case, examples of atomic formulas are $x=z$, $x <_P y$ or $x <_V x$.
First-order formulas are then obtained
inductively from the atomic formulas, as combinations of smaller formulas using
the usual connectives of the first-order logic: negation ($\neg$), conjunction ($\wedge$), 
disjunction ($\vee$), implication ($\Rightarrow$), equivalence ($\Leftrightarrow$),
universal and existential quantification ($\forall x \, \phi$ or $\exists x\, \phi$,
where $x$ is a variable and $\phi$ a formula).
A sentence is a formula that has no free variable, that is to say in which all
variables are quantified. 
For example, $ \exists x\, \exists y\,  (x <_P y\, \wedge\, y <_V x)$
is a first order sentence.

First order sentences describe properties of the models, in our case of permutations.
We say that a permutation $\sigma$ {\em satisfies} a sentence $\Psi$,
and write $\sigma \models \Psi$,
if $\Psi$ holds
true when the variables are interpreted as elements of $A^\sigma$ and when the symbols $<_P$
and $<_V$ are interpreted as $<^\sigma_P$ and $<^\sigma_V$. For example, $\sigma \models \exists x\, \exists y\,  (x <_P y\, \wedge\, y <_V x)$ precisely if $\sigma$ contains two elements that form a $21$ pattern.

\subsection{Logical types}
We start by recalling the notion of \emph{quantifier depth} of a first-order formula.
Informally, this is the maximal number of nested quantifiers in the formula.
Formally, we can define it recursively as follows. If $\Psi$ is an atomic formula
(such as $u=v$, $x<_V y$ or $z <_P t$),
then $\qdepth(\Psi) = 0$. Otherwise:
\begin{align*}
\qdepth(\neg \Psi) &= \qdepth (\Psi), \\
\qdepth(\Psi \vee \theta) = \qdepth(\Psi \wedge \theta) &= \max (\qdepth(\Psi), \qdepth(\theta)), \\
\qdepth(\exists x \, \Psi) = \qdepth(\forall x \, \Psi) &= \qdepth(\Psi) + 1. %\quad \mbox{(where $x$ is free in $\Psi$).}
\end{align*}

Fix $k \ge 2$ and consider first-order {\em sentences} of quantifier depth at most $k$.
We consider two first-order sentences to be equivalent if they are satisfied 
by the same set of permutations.
By putting formulas in, say, prenex conjunctive normal forms, we see that, in any theory with finite signature, for each fixed $k\ge 2$,
there are finitely many first-order sentences of quantifier depth at most $k$,
up to equivalence.

We say that two permutations $\sigma$ and $\tau$ are $k$-equivalent,
and write $\sigma \equiv_k \tau$, if they are models for the same first-order
sentences of quantifier depth $k$.
For each fixed $k$, this relation splits the set of permutations into 
{\em finitely many} equivalence classes.
These classes are called \emph{logical types of order $k$}; 
their set is denoted $\mathcal T_k$.

\subsection{Ehrenfeucht-Fraïssé games}
We will make use here of a fundamental result of finite model theory,
relating satisfaction of first-order sentences to a combinatorial game.
We present here this result in the context of permutations (in the $\TOTO$ logic).
We refer to~\cite{albert2020two} for a more detailed specific discussion on permutations
and to~\cite{Gradel2007Finite-model-th} for a general reference on finite model theory.

Let $\alpha$ and $\beta$ be two permutations, and let $k$ be a positive integer. The \emph{Ehrenfeucht-Fra\"{i}ss\'{e} (EF) game of length $k$} played on $\alpha$ and $\beta$ is a game between two players (named \emph{Duplicator} and \emph{Spoiler}) according to the following rules:
\begin{itemize}
\item
The players alternate turns, and Spoiler moves first.
\item
The game ends when each player has had $k$ turns.
\item
At his $i\th$ turn, ($1 \leq i \leq k$) Spoiler chooses either an element $a_i \in \alpha$ or an element $b_i \in \beta$. In response, at her $i\th$ turn, Duplicator chooses an element of the other permutation.
Namely, if Spoiler has chosen $a_i \in \alpha$, then Duplicator chooses an element $b_i \in \beta$, 
and if Spoiler has chosen $b_i \in \beta$, then Duplicator chooses $a_i \in \alpha$.
\item
At the end of the game if the map $a_i \mapsto b_i$ for all $i \le k$
preserves both position and value orders, then Duplicator wins.
Otherwise, Spoiler wins.
\end{itemize}

The connection between EF games and quantifier depth is captured
 in the fundamental theorem of Ehrenfeucht and Fra\"{i}ss\'{e},
 which we state here for permutations.
\begin{proposition}
Two permutations $\alpha$ and $\beta$ are $k$-equivalent
if and only if Duplicator has a winning strategy in the EF game of length $k$ played on $\alpha$ and $\beta$.
\end{proposition}
We will use this result below to prove that 
certain operations on permutations affect the logical types in a prescribed manner; see Lemma~\ref{lem:decompo_logic}.

\subsection{Algebraic systems of equations and the Drmota-Lalley-Woods theorem}
\label{sec:DLW}
As mentioned in the introduction, the proof of the logical limit law 
relies on techniques from analytic combinatorics,
which we now introduce.

Consider a system of equations: for $1 \le i \le d$, one has
\begin{equation}\label{eq:system_general}
  y_i(z)= \Phi_i(z,y_1(z),\cdots,y_d(z)),
\end{equation}
where the $y_i$ are unknown formal power series in $z$
and each $\Phi_i$ is a given formal power series in $z,y_1,\cdots,y_d$ with
{\em non-negative} coefficients.

To such a system, we associate its {\em dependency graph}, which is a
directed graph on the vertex set $\{1,\cdots,d\}$ with the following edges:
there is an edge from $i$ to $j$ if $y_i$ appears in the equation
defining $y_j$, i.e.~if $\frac{\partial \Phi_j}{\partial y_i} \ne 0$.

We recall that a directed graph is said to be {\em strongly connected}
if, for any pair of vertices $(u,v)$, there is an oriented path from $u$ to $v$.
In general, one can consider the {\em strongly connected}  components
of a directed graph $G$. These are maximal induced subgraphs that are strongly connected.
A strongly connected component $C$ of $G$ is said to be {\em terminal},
if there does not exist $v$ in $C$ and $w$ outside $C$ with an oriented edge from $v$ to $w$.

To illustrate these definitions, let us consider the system
\begin{align*}
  y_1&=1+zy_1 y_4 y_5 +z y_2^2;\quad y_2=z y_1 y_5; \quad y_3=1;\\
  y_4&=1+z y_4 +z^2 y_3 y_6; \quad y_5 = z^2 y_2^2 y_4; \quad y_6= z y_4^2.
\end{align*}
Its dependency graph is drawn in Fig.~\ref{fig:dependency}. 
The strongly connected components are represented with a dashed contour,
and the terminal one (which is unique in this case) has a light gray background.
\begin{figure}
  \begin{center}
    \includegraphics{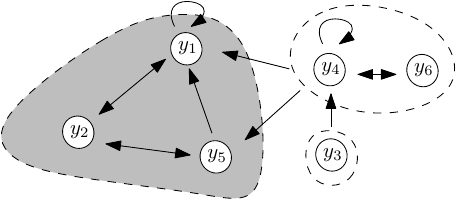}
  \end{center}
  \caption{Example of dependency graph of an algebraic system of equations}
  \label{fig:dependency}
\end{figure}

Finally,
we recall that a series $A=\sum_{n \ge 0} a_n z^n$ is said to be periodic if 
there exists $d\geq2$ and $r$
such that $a_n=0$ unless $n \in d \mathbb Z +r$. 
It is aperiodic otherwise.

In analytic combinatorics, we are interested in the behavior of 
generating series of combinatorial objects around their dominant singularities.
It turns out that solutions of a system of equations such as \eqref{eq:system_general}
have some specific behavior under rather general hypotheses.
This is known as the Drmota--Lalley--Woods theorem.
We use the standard Kronecker symbol $\delta_{i,j}$ defined by $\delta_{i,j} = 1$ if $i=j$ and $0$ otherwise.
\begin{proposition}
 Consider a system as in \eqref{eq:system_general}, where the $\Phi_i$ have nonnegative coefficients, and assume that:
 \begin{enumerate}%[label=(\roman*)]
   \item the system is nonlinear in the $y_i$'s, {\em i.e.}~there exist indices $i,j,k$ (possibly with repetitions)
     such that $\frac{\partial^2 \Phi_k}{\partial y_i \partial y_j}(z;y_1,\cdots,y_d) \ne 0$;
   \item for each $i$, one has $\Phi_i(0;y_1,\cdots,y_d)=0$
(as power series in $y_1, \ldots, y_d$);
\item there exist $j$ and $k$ such that $\Phi_j(z;0,\cdots,0) \ne 0$ and $\frac{\partial \Phi_k}{\partial z}(z;y_1,\cdots,y_d) \ne 0$ (as power series in $z$ and in $z,y_1,\ldots,y_d$ respectively);
\item the dependency graph of the system is strongly connected;
\item each $\Phi_i$ is convergent for $(r,s_1,\cdots,s_d)$ in a neighbourhood of $(0,\cdots,0)$
  and the intersection of their regions of convergence 
  contains a solution $(r,s_1,\cdots,s_d)$ of the system
  \begin{equation}\label{eq:characteristic_system}
    \begin{cases}
      s_i=\Phi_i(r,s_1,\cdots,s_d) &\text{ for all $i\le d$;}\\
      0=\det\left(\delta_{i,j} - \frac{\partial \Phi_i}{\partial y_j}(r,s_1,\cdots,s_d)\right)_{1 \le i,j \le d}
    \end{cases}
  \end{equation}
\item at least one of the series $y_i$ is aperiodic.
 \end{enumerate}
 Then the system \eqref{eq:system_general} has a unique solution with $y_1(0)=\dots=y_d(0)=0$.
 This solution satisfies that all the $y_i$ have the same radius of convergence $\rho$,
 which is the first coordinate $r$ of the solution of \eqref{eq:characteristic_system}.
 Moreover, for each $i$, there is a {\bf nonzero} constant $A_i$ such that
 \begin{equation}\label{eq:conclusion_DLW}
   [z^n]y_i \sim A_i \,\rho^{-n}\, n^{-3/2}.
 \end{equation}
  \label{thm:DLW}
\end{proposition}
Some bibliographic comments are in order.
The name Drmota--Lalley--Woods theorem is given in \cite{flajolet2009analytic}
(see Theorem VII.6 there),
but this reference only treats the case of polynomial systems of equations.
For the more general case of analytic equations, which we will need in this paper,
we refer to \cite[Theorem 2.33]{DrmotaRandomTrees}.
Note that we only consider a special case of \cite[Theorem 2.33]{DrmotaRandomTrees},
where we do not consider any auxiliary formal variables $u_i$.
Also, this reference gives a singular expansion
of the generating series $y_i$ around the singularity
and ensures (in the aperiodic case, as assumed above) the analyticity of $y_i$
on a $\Delta$-domain, so we need to apply the so-called transfer theorem
\cite[Theorem VI.4]{flajolet2009analytic} to get \eqref{eq:conclusion_DLW} as above.
The fact that $A_i \ne 0$ is a consequence of the property $h_j(x,\bm u)\ne 0$
given in \cite[Theorem 2.33]{DrmotaRandomTrees}.

\section{Proof of Theorem~\ref{thm:ConvLaw_Av231}}
The goal of this section is to prove Theorem~\ref{thm:ConvLaw_Av231},
in particular to prove that a logical limit law holds for the class $\mathcal C=\Av(231)$
of 231-avoiding permutations.
We use the convention that $\mathcal C$ contains the empty permutation of size $0$.

We first present a standard recursive construction of the elements of $\mathcal C$.
To this end, recall that the direct sum $\tau \oplus \pi$ of two permutations $\tau$ and $\pi$
is obtained by juxtaposing $\tau$ and $\pi$ (in one-line notation) and increasing
all values in $\pi$ by the size of $\tau$. The skew sum $\tau \ominus \pi$ is defined
similarly except that, this time, values in $\tau$ are shifted by the size of $\pi$.
For example, $12 \oplus 231 = 12453$, while   $12 \ominus 231 = 45231$.
It is well-known -- see {\em e.g.}~\cite[Chapter 4]{BonaBook} --
that a non-empty permutation $\sigma$ in $\mathcal C$ can be uniquely
decomposed as $\tau \oplus (1 \ominus \pi)$, for some (possibly empty)
 $\tau$ and $\pi$ in $\mathcal C$.
This yields the following equation for the generating series
$C(z)=1+z\, C(z)^2$, whose unique non-negative power series solution is given by $C(z)=(1-\sqrt{1-4z})/(2z)$.
Therefore, as it is well known, $[z^n]C(z)$ is the $n$-th Catalan number
and one has the asymptotics $[z^n]C(z) \sim \pi^{-1/2} \, 4^n\, n^{-3/2}$.

\subsection{Refining the combinatorial equation.} 
From now on, we fix $k \ge 2$.
The decomposition $\sigma=\tau \oplus (1 \ominus \pi)$ is {\em compatible with logical types}
in the following sense.
\begin{lemma}
\label{lem:decompo_logic}
Let $\tau$ and $\pi$ be permutations and let $t_1$ and $t_2$ be their logical types of order $k$.
Then the logical type of order $k$ of $\tau \oplus (1 \ominus \pi)$ depends only on $t_1$ and $t_2$;
we denote it by $H_k(t_1,t_2)$.
\end{lemma}
\begin{proof}
  Consider permutations $\tau_1$, $\tau_2$, $\pi_1$ and $\pi_2$
  such that  $\tau_1 \equiv_k \tau_2$ and $\pi_1 \equiv_k \pi_2$.
  We set $\sigma_1=\tau_1 \oplus (1 \ominus \pi_1)$,
  $\sigma_2=\tau_2 \oplus (1 \ominus \pi_2)$ and want to show that 
  $\sigma_1 \equiv_k \sigma_2$.

  We consider the $k$-round Ehrenfeucht-Faïssé game on the board $(\sigma_1,\sigma_2)$.
  We recall that, since $\tau_1 \equiv_k \tau_2$ and $\pi_1 \equiv_k \pi_2$,
  Duplicator has winning strategies on the boards $(\tau_1,\tau_2)$ and $(\pi_1,\pi_2)$.
  Using those, she has the following strategy on $(\sigma_1,\sigma_2)$:
  \begin{itemize}
    \item 
  if Spoiler selects the maximum of 
  $\sigma_i$ for $i=1$ or $2$ (the maximum corresponds to $1$
  in the decomposition $\sigma_i=\tau_i \oplus (1 \ominus \pi_i)$),
  then Duplicator selects the maximum on the other side;
\item if Spoiler selects an element of $\tau_1$ (resp. $\tau_2$), 
  then Duplicator selects an element of $\tau_2$ (resp. $\tau_1$),
  according to the winning strategy for the game on $(\tau_1,\tau_2)$;
\item similarly with $\pi_1$ and $\pi_2$.
  \end{itemize}
  It is easy to see that this is a winning strategy for Duplicator,
  proving that $\sigma_1 \equiv_k \sigma_2$.
\end{proof}

For $t$ a logical type in $\mathcal T_k$,
we write $C_t$ for the generating series of 231-avoiding permutations of type $t$.
We remove from $\mathcal T_k$ the logical types $t$ such that $C_t \equiv 0$.  
Since $\sum_t C_t(z)=C(z)$ and since $C(z)$ is convergent at its radius of convergence $R=1/4$,
all $C_t(z)$ are convergent for $|z| \le 1/4$ (but the radius of convergence of some of them 
might be larger than $1/4$).

There is a special element of $\mathcal T_k$, the type of the empty permutation (which is alone in its class),
denoted by $\varnothing$.
Using Lemma~\ref{lem:decompo_logic}, the equation $C(z)=1+z\, C(z)^2$ can be refined as a system
\begin{equation}
\text{for all }t \in \mathcal T_k,\quad
  C_t(z) = \delta_{t,\varnothing} + z\sum_{t_1,t_2 \in\mathcal T_k \atop H_k(t_1,t_2)=t} C_{t_1}(z)\, C_{t_2}(z)
=:F_t(z;C_u, u \in \mathcal T_k).
\label{eq:systeme_Ct}
\end{equation}

The key point in the proof of the logical limit law in Theorem~\ref{thm:ConvLaw_Av231}
consists in extracting from this system the asymptotic behavior of the
function $C_t(z)$.
Though far from being explicit, this system has noticeable properties.
Let us consider the dependency graph $G_k$ of this system of equations.
We claim that $G_k$ has a unique terminal strongly connected component.
Indeed suppose for the sake of contradiction that $u_1$ and $u_2$ are elements
of two different strongly connected components $S$ and $S'$ of $G_k$.
Then, letting $t=H(u_1,u_2)$, by construction there is an edge from $u_1$ to $t$ in $G_k$,
as well as an edge from $u_2$ to $t$.
Since $S$ (resp.~$S'$) is strongly connected, this implies that $t$ is in $S$ (resp.~in $S'$).
But $t$ cannot be simultaneously in $S$ and $S'$, whence a contradiction.

Thus $G_k$ has a unique terminal strongly connected component,
which we denote $G_k^\star$.
We let $\mathcal T_k^\star$ be the set of vertices of $G_k^\star$,
i.e.~the subset of types which are in the terminal strongly connected component.
Furthermore, we write $\mathcal T_k^\bullet=\mathcal T_k \setminus \mathcal T_k^\star$.
\medskip

Another easy-to-establish property is the following lemma, regarding aperiodicity. It will be useful when applying the Drmota-Lalley-Woods theorem in~Section~\ref{sec:asymptoticAnalysis}. 

\begin{lemma}
\label{lem:aperiodic}
All series $C_t$ for $t$ in $\mathcal T_k^\star$ are aperiodic.
\end{lemma}
\begin{proof}
A standard EF game argument -- see \cite[Propositions 24 and 26]{albert2020two} --
 asserts that there exists $K>0$ such that
all permutations $12 \cdots n$ for $n \ge K$ have the same logical type of order $k$, which we will denote by~$t_\nearrow$. Clearly, $C_{t_\nearrow}$ is aperiodic.

Furthermore, it is easy to see that aperiodicity propagates along edges of the dependency graph.
More formally, we claim that
if $C_u$ is aperiodic for some type $u$,
and if $\frac{\partial F_t}{\partial C_u} \ne 0$, then $C_t$ is aperiodic.
Indeed, in this case, $F_t$ contains a monomial $z C_u C_v$ for some $v$.
Since $C_v \neq 0$, there exists $r$ such that $C_t \ge z^r C_u$ coefficient-wise,
implying that $C_t$ is aperiodic.

Starting from $C_{t_\nearrow}$, we can follow outgoing edges of the dependency graph
 until we reach a state $t_0$ in the terminal
strongly connected component. Then $C_{t_0}$ is aperiodic.
Using again the propagation of aperiodicity along edges,
we conclude all series $C_t$, for $t \in \mathcal T_k^\star$ are aperiodic. 
\end{proof}

\subsection{The Jacobian matrix and its spectral radius.}
We consider the Jacobian matrix of the system:
\[M_k(z) = \left( \frac{\partial F_t}{\partial C_u}(z;C_v(z),v \in \mathcal T_k) \right)_{t,u \in \mathcal T_k}.\]
Its rows and columns are indexed by elements of $\mathcal T_k$.
To write down the matrix, we order $\mathcal T_k$ such that the elements of $\mathcal T_k^\star$ come first.

Let $t$ and $u$ be in $\mathcal T_k^\bullet$ and $\mathcal T_k^\star$ respectively.
Since $u$ is in the  terminal component of $G_k$, while $t$ is not, there cannot be an edge from $u$ to $t$.
Hence, by construction, one has 
$ \frac{\partial F_t}{\partial C_u}=0$.
This implies that the matrix $M_k$ then decomposes into blocks as
\begin{equation}
  M_k= \begin{pmatrix}
  M_k^\star & *\\
  0 & M_k^\bullet
  \label{eq:BlockDecompoM}
\end{pmatrix},\end{equation}
where $M_k^\star$ and $M_k^\bullet$ are the Jacobian matrices restricted to $\mathcal T_k^\star$
and $\mathcal T_k^\bullet$ respectively, $0$ is the zero matrix, and $*$ denotes an unknown matrix.
\medskip

For a square matrix $A$ we denote $\SR(A)$ its spectral radius, i.e.~the maximum modulus
of an eigenvalue of $A$.
The following lemma will be useful in Section~\ref{sec:asymptoticAnalysis} for  finding the radii of convergence of the series that are solutions of our system \eqref{eq:systeme_Ct}. 
%By Perron-Frobenius theorem, if $A$ has nonnegative entries, then $\SR(A)$ is an eigenvalue of $A$.
\begin{lemma}
  \label{lem:SpectralRadius_MBullet}
  We have $\SR\big(M_k^\bullet(1/4)\big)<\SR\big(M^\star_k(1/4)\big)=\SR\big(M_k(1/4)\big)=1$.
\end{lemma}
\begin{proof}
Consider the column sums of $M_k(z)$, where $|z| \le 1/4$: for $u$ in $\mathcal T_k$, we have
\[\sum_{t \in \mathcal T_k} M_k(z)_{t,u} = \frac{\partial \left( \sum_{t \in \mathcal T_k} F_t \right)}
    {\partial C_u}(z;C_v(z),v \in \mathcal T_k). \]
But by construction
\[\sum_{t \in \mathcal T_k} F_t(z;C_v,v \in \mathcal T_k) = 1 + z \, \left(\sum_{v \in \mathcal T_k} C_v \right)^2,\]
so that we get
\[\sum_{t \in \mathcal T_k} M_k(z)_{t,u} = 2 z\,\left(\sum_{v \in \mathcal T_k} C_v(z) \right) = 2 z \, C(z). \]
The right-hand side does not depend on $u$, i.e.~$M_k(z)$ has constant column sums. 
For $z=1/4$, this sum is $2 \cdot (1/4) \cdot C(1/4)=1$.
Thus $1$ is an eigenvalue of $M_k(1/4)$, proving $\SR(M_k(1/4)) \ge 1$.
On the other hand, the spectral radius of a nonnegative matrix is at most its maximal column sum
-- see, {\em e.g.}, \cite[Lemma 4.4]{woods1997coloring} --,
proving $\SR(M_k(1/4)) =1$.

The same argument, together with the fact that the lower left block of $M_k$
is filled with zeroes, proves that $\SR(M^\star_k(1/4)) =1$.

We now consider $M_k^\bullet(1/4)$, which appears as a block in $M_k(1/4)$ (see decomposition \eqref{eq:BlockDecompoM}).
Fix an arbitrary element $u^\star$ in $\mathcal T_k^\star$.
For each $u$ in $\mathcal T_k^\bullet$, we set $t_u=H_k(u,u^\star)$.
Note that $t_u \in \mathcal T_k^\star$. 
By construction, we have $M_k(1/4)_{t_u,u}>0$.
This implies that the column indexed by $u$ in $M_k(1/4)$ has a nonzero element outside
the block $M_k^\bullet(1/4)$.
Consequently, the column sums of $M_k^\bullet(1/4)$ are smaller than those of $M_k(1/4)$,
i.e., for $u$ in $\mathcal T_k^\bullet$, we have
\[\sum_{t \in \mathcal T_k^\bullet} M_k^\bullet(1/4)_{t,u} < \sum_{t \in \mathcal T_k} M_k(1/4)_{t,u}=1.\]
Using that the spectral radius is at most the maximal column sum, we get
$\SR(M_k^\bullet(1/4)) <1$.
\end{proof}

\subsection{Radius of convergence and asymptotic analysis.}\label{sec:asymptoticAnalysis}

\begin{lemma}
\label{lem:RCv_Bullet}
  For $t$ in $\mathcal T_k^\bullet$, $C_t$ has radius of convergence strictly larger than $1/4$.
  Consequently, there exists $\kappa_t<4$ such that $[z^n] C_t(z) =O(\kappa_t^n)$.
\end{lemma}
\begin{proof}
  Let $t$ be in $\mathcal T_k^\bullet$
  and consider the equation $C_t(z) =:F_t(z;C_u, u \in \mathcal T_k)$ in the system \eqref{eq:systeme_Ct}.
  As explained above (see the text above Eq.~\eqref{eq:BlockDecompoM}), this equation only involves
  series $C_u$, for $u$ in $\mathcal T_k^\bullet$ (and not those for $u$ in $\mathcal T_k^\star$).

 We can therefore consider the restriction of the system \eqref{eq:systeme_Ct} to
the variables $(C_t, t \in \mathcal T_k^\bullet)$:
\begin{equation}
 \text{for all }t \in \mathcal T_k^\bullet, \quad  C_t(z) 
=F_t(z;C_u, u \in \mathcal T_k^\bullet).
\label{eq:systeme_Ct_bullet}
\end{equation}
The Jacobian matrix of this system at $z=1/4$ is $M_k^\bullet(1/4)$.
From Lemma~\ref{lem:SpectralRadius_MBullet}, the matrix $(\Id-M_k^\bullet(1/4))$ is invertible.
Therefore, using the multivariate implicit function theorem
-- see \cite[Lemma 5.1]{woods1997coloring} or \cite[Theorem B.6]{flajolet2009analytic} --,
 Eq.~\eqref{eq:systeme_Ct_bullet} has a unique solution for $z$ in a neighbourhood $V$ of $1/4$,
 and this solution defines analytic functions $C_t$ for $t$ in $\mathcal T_k^\bullet$.

We recall that $C_t(z)$ is analytic for $|z| < 1/4$ for all $t$ in $\mathcal T_k$,
since $C_t$ is dominated coefficient-wise by the Catalan series $C(z)$.
The above result means that for  $t$ in $\mathcal T_k^\bullet$,
there is an analytic extension of $C_t$ in a neighbourhood of $1/4$,
i.e.~$1/4$ is not a singularity of $C_t$.
Since $C_t$ has nonnegative coefficients,
Pringsheim's theorem applies \cite[Theorem IV.5]{flajolet2009analytic},
and we conclude that $C_t$ has a radius of convergence larger than 1/4 
(for $t$ in $\mathcal T_k^\bullet$).

The consequence on the growth of the coefficients $[z^n] C_t(z)$
is standard; see, {\em e.g.}, \cite[Theorem IV.7]{flajolet2009analytic}.
\end{proof}

\begin{lemma}
\label{lem:asymp_Ctn}
  For $t$ in $\mathcal T_k^\star$, there exists a constant $A_t>0$ such that, as $n$ tends to $+\infty$, one has
  \begin{equation}
   \label{eq:asymp_Ctn}
   [z^n] C_t(z) \sim A_t \, 4^n\, n^{-3/2}.
     \end{equation}
\end{lemma}
\begin{proof}
We consider the system \eqref{eq:systeme_Ct} as a system of equations for the series
$(C_u; u \in \mathcal T_k^\star)$, seeing the series $(C_v; v \in \mathcal T_k^\bullet)$ as parameters.
More formally for $t \in \mathcal T_k^\star$, we let $\Phi_t(z; C_u, u \in \mathcal T_k^\star)$
be the function $F_t(z; C_u, u \in \mathcal T_k)$ where the indeterminates 
$(C_u, u \in \mathcal T_k^\star)$ are untouched while
the $(C_v, v \in \mathcal T_k^\bullet)$ are substituted by their actual value $C_v(z)$.
Then $(C_u(z); u \in \mathcal T_k^\star)$ are the unique formal power series solutions of the system
\begin{equation}
\label{eq:systeme_Ct_critique}
\text{for all }t\in  \mathcal T_k^\star, \quad C_t=\Phi_t(z; C_u, u \in \mathcal T_k^\star).
\end{equation}

The dependency graph of this system is strongly connected,
and we will apply the Drmota-Lalley-Woods theorem\footnote{We note that, even though the original system
\eqref{eq:systeme_Ct} is polynomial, the restricted system \eqref{eq:systeme_Ct_critique}
after substitution of the $(C_u)_{u \in \mathcal T_k^\bullet}$, is not polynomial any more
since some of the $C_u$ might be infinite series.
This is the reason why we need the general version of Drmota--Lalley--Woods theorem with analytic equations,
and not only the one for polynomial systems presented in \cite{flajolet2009analytic}.}
recalled in Section~\ref{sec:DLW}.
Conditions 1, 2, 3 and 4 of Proposition~\ref{thm:DLW} are easy to check.
Condition 6 follows from Lemma~\ref{lem:aperiodic}.
It remains to check condition 5.

The functions $\Phi_t$ are power series with nonnegative integer coefficients
and are analytic on the region
\[ \{|z| <\rho_2, (C_u) \in \mathbb R_+^{\mathcal T_k^\star}\},\]
where $\rho_2$ is the minimal radius of convergence of a series $C_v(z)$ with $v \in \mathcal T_k^\bullet$ (recall that $\Phi_t$ now depends on $z$ through the substituted series
$C_v(z)$ with $v \in \mathcal T_k^\bullet$).
From Lemma~\ref{lem:RCv_Bullet}, we have $\rho_2>1/4$.
We recall that all series $(C_u)_{u \in \mathcal T_k^\star}$ are convergent at $z=1/4$
(since they are coefficient-wise dominated by $C$).
The point $(1/4,(C_u(1/4))_{u \in \mathcal T_k^\star})$ therefore lies in the analyticity
region of the functions $\Phi_t$.
We claim that $(1/4,(C_u(1/4))_{u \in \mathcal T_k^\star})$ is a solution of the system
\eqref{eq:characteristic_system}.
That it satisfies $s_i=\Phi_i(r,s_1,\cdots,s_d)$ is clear,
since the $(C_u)_{u \in \mathcal T_k^\star}$ satisfy \eqref{eq:systeme_Ct_critique}.
Besides, the equality 
\[      0=\det\left(\delta_{i,j} - \frac{\partial \Phi_i}{\partial y_j}(r,s_1,\cdots,s_d)\right)_{1 \le i,j \le d}\]
is implied by Lemma~\ref{lem:SpectralRadius_MBullet} (the Jacobian matrix of the restricted
system is $M_k^\star$).
We conclude that condition 5 of Proposition~\ref{thm:DLW} is satisfied as well.

Therefore Proposition~\ref{thm:DLW} applies. All series $(C_u(z))_{u \in \mathcal T_k^\star}$
have the same radius of convergence $\rho=1/4$
and \eqref{eq:asymp_Ctn} holds.
\end{proof}

\subsection{Concluding the proof of the logical limit law}

Let $\Psi$ be a first-order sentence on permutations,
and denote by $k$ its quantifier depth.
 Then there exists a subset $\mathcal T_{\Psi}$ of $\mathcal T_k$ such that
 \[ \{ \sigma \in \mathcal C: \sigma \models \Psi\} = \biguplus_{t \in \mathcal T_{\Psi}} \mathcal C_t.\]
 This implies, for $n \ge 1$ and $ \sigma_n$ a uniform random 231-avoiding permutation of size $n$,
 \[\PP( \sigma_n \models \Psi) = \sum_{t \in \mathcal T_{\Psi}} \frac{[z^n] C_t(z)}{[z^n] C(z)}.\]
Since the set $\mathcal T_k$ of $k$-logical types of permutations is finite, the above sum is finite.
The existence of a limit then follows from Lemmas~\ref{lem:RCv_Bullet} and~\ref{lem:asymp_Ctn},
recalling that $[z^n] C(z)\sim\pi^{-1/2}\, 4^n\, n^{-3/2}$.
This proves the logical limit law.

Item 1 in Theorem~\ref{thm:ConvLaw_Av231} also follows immediately.
 If $\mathcal T_{\Psi}$ contains at least one type in $\mathcal T_k^\star$,
 then the limit of $\PP( \sigma_n \models \Psi)$ is positive.
 On the other hand, if $\mathcal T_{\Psi} \cap \mathcal T_k^\star=\emptyset$, 
 then $\PP( \sigma_n \models \Psi)$ decreases exponentially fast to $0$ by Lemma~\ref{lem:RCv_Bullet}.
 
\subsection{Set of limiting probabilities}
We consider here the set of limiting probabilities
\[L:=\big\{\lim_{n \to +\infty} \PP(\sigma_n \models \Psi), \Psi\text{\ FO sentence} \big\},\]
and we want to prove that it is dense in $[0,1]$,
which is the remaining part of Theorem~\ref{thm:ConvLaw_Av231}.

We start by recalling a result of Kakeya on the set of subsums of a convergent series,
see \cite{KakeyaSubsums} for the original statement and \cite{KakeyaProof} for a complete proof. We only copy here a part of the theorem, which is the one relevant for us.
\begin{lemma}
\label{lem:Kakeya}
Let $(p_i)_{i \ge 0}$ be a non-increasing sequence of positive real numbers
such that $\sum_{i\ge 0} p_i<+\infty$.
Assume that for all $i \ge 0$, one has $p_i \le \sum_{j>i} p_j$.
Then 
\[\left\{\sum_{i \ge 0} \eps_i p_i; (\eps_i) \in \{0,1\}^{\mathbb N}\right\}
    = \left[0,\sum_{i \ge 0} p_i \right],\]
    i.e.~the set of (finite and infinite) subsums of $\sum_{i\ge 0} p_i$ is the whole
    interval $\left[0,\sum_{i \ge 0} p_i \right]$.
\end{lemma}

 In the following, we use the notation $\Cat_n=\frac{1}{n+1}\binom{2n}n$ for the $n$-th
Catalan number. Let $ \sigma_n$ be
  a uniform random $231$-avoiding permutation of size $n$.
We can decompose $ \sigma_n$ as $ \sigma_n= \tau_n \oplus(1\ominus  \pi_n)$. Note that $ \tau_n$ and $ \pi_n$ are random $231$-avoiding permutations, and that their sizes themselves are random.
 Their asymptotic distribution is given as follows.
 \begin{lemma}
 \label{lem:Asymptotique_Tau}
 Fix a $231$-avoiding permutation $\rho$. Then we have
 \[ \lim_{n \to +\infty} \PP( \tau_n=\rho)=\lim_{n \to +\infty} \PP( \pi_n=\rho) = 4^{-|\rho|-1}.\]
 \end{lemma}
 \begin{proof}
Let $k$ be the size of $\rho$. There are exactly $ \Cat_{n-k-1}$ permutations $\sigma$ in $\Av_n(231)$
of the form $\rho \oplus(1\ominus \pi)$: indeed, one can choose $\pi$ freely in $\Av_{n-k-1}(231)$.
Therefore $$\PP( \tau_n=\rho)=\frac{ \Cat_{n-k-1}}{ \Cat_n},$$ and the limit is $4^{-k-1}$ as claimed.
The equality $\lim_{n \to +\infty} \PP( \pi_n=\rho) = 4^{-|\rho|-1}$ is proved similarly.
\end{proof}

Let $\mathcal F$ and $\mathcal F'$ be two {\bf finite} subsets of $\Av(231)$.
We consider the event 
\[ E_{\mathcal F,\mathcal F'}: \ ( \tau_n \in \mathcal F) \vee ( \pi_n \in \mathcal F'). \]
Clearly, $E_{\mathcal F,\mathcal F'}$ is a first order property.
For $n$ large enough, the events $ \tau_n \in \mathcal F$ and $ \pi_n \in \mathcal F'$ are incompatible
since $| \tau_n|+| \pi_n|=n-1$. We therefore have, using Lemma~\ref{lem:Asymptotique_Tau}
\[ \lim_{n\to+\infty} \PP( E_{\mathcal F,\mathcal F'}) = 
\sum_{k \le K} (|\mathcal F_k| + |\mathcal F'_k|) \, 4^{-k-1},\]
where $\mathcal F_k$ (resp. $\mathcal F'_k$) is the set of permutations of size $k$ in $\mathcal F$ (resp. $\mathcal F'$),
and where $K$ is the maximal size of a permutation in either $\mathcal F$ or $\mathcal F'$.
For each $k \le K$, the quantity $|\mathcal F_k| + |\mathcal F'_k|$ can take any value between
$0$ and $2 \Cat_k$, so that $L$ contains the set 
\[L':=\left\{ \sum_{k= 0}^K a_k 4^{-k-1}, \ K \ge 0\text{ and } 0 \le a_k \le 2 \Cat_k\right\}.\]
Let $(p_i)_{i\ge 0}$ be the non-increasing sequence containing $4^{-k-1}$ exactly $2\Cat_k$
times for each $k\ge 0$, and no other elements.
We have
\[\sum_{i\ge 0} p_i= \sum_{k= 0}^{+\infty} 2 \Cat_k 4^{-k-1} = \tfrac12 C(\tfrac14)=1,\]
where we recall that $C(z)=(1-\sqrt{1-4z})/(2z)$ is the Catalan generating series.
On the other hand,
\[L' = \left\{ \sum_{i \in J} p_i: \ J \subset \mathbb{N}, |J| <+\infty\right\}.\]
In words, $L'$ is the set of finite subsums of $\sum_{i\ge 0} p_i$.
Its topological closure contains the set $L''$ of all (finite or infinite) subsums of $\sum_{i\ge 0} p_i$.
Observe finally that $p_i \le \sum_{j>i} p_j$ for all $i$.
Applying Lemma~\ref{lem:Kakeya}, we have $L''=[0,1]$,
concluding the proof of Theorem~\ref{thm:ConvLaw_Av231}.

 \section*{Acknowledgements}
 VF is partially supported by the Future Leader program of the LUE initiative 
 (Lorraine Université d'Excel\-lence).
 MA, MB and VF are grateful to the Department of Mathematics of the Universitat Politècnica de Catalunya for their visit in the Fall 2016 during which this research started. 
 Finally,
 the authors warmly thank anonymous referees, whose comments helped them improve
 the presentation of the paper.

\bibliography{biblio.bib}

\end{document}

%% file: pic.txt
\[\begin{tikzpicture}[scale=0.6]
  \draw[gray,very thin] (0,0) grid (5,5);
  \fill (0.5,2.5) circle (0.15) node[below right =-0.1] {\scriptsize $A$};
  \fill (1.5,0.5) circle (0.15) node[below right =-0.1] {\scriptsize $B$};
  \fill (2.5,3.5) circle (0.15) node[below right =-0.1] {\scriptsize $C$};
  \fill (3.5,4.5) circle (0.15) node[below right =-0.1] {\scriptsize $D$};
  \fill (4.5,1.5) circle (0.15) node[below right =-0.1] {\scriptsize $E$};
\end{tikzpicture}\]